\documentclass[10pt,twoside ]{article}  % do not change
\usepackage{a4,times,amsmath,theorem,latexsym,amssymb}
\advance\hoffset by -0.3cm

\newcommand{\SDpagebegin}{901}
\newcommand{\SDpageend}{9999}

\makeatletter
\def\@maketitle{%
  \newpage
  {\vspace*{-11ex}%
  \small\noindent\begin{tabular}{@{}l}
  \end{tabular}}
  \null
  \vskip 5em%
  \begin{flushleft}%
    {\LARGE \bf\@title \par}% war urspruenglich \Huge \@title, htb
    \vskip 1.5em%
       {\large\bf \lineskip .5em \@author \par}%
    \vskip 1.0em%
  \end{flushleft}%
  \par
  \vskip 1.5em}
\renewcommand\section{\@startsection {section}{1}{\z@}%
                                   {-3.5ex \@plus -1ex \@minus -.2ex}%
                                   {1.5ex \@plus.2ex}%
                                   {\normalfont\Large\bfseries}}
\renewcommand\subsection{\@startsection{subsection}{2}{\z@}%
                                     {-3.0ex\@plus -2ex \@minus -.2ex}%
                                     {1.0ex \@plus .2ex}%
                                     {\normalfont\large\bfseries}}
\renewcommand\subsubsection{\@startsection{subsubsection}{3}{\z@}%
                                     {-2.5ex\@plus -2ex \@minus -.2ex}%
                                     {0.9ex \@plus .2ex}%
                                     {\normalfont\normalsize\bfseries}}
\renewcommand\paragraph{\@startsection{paragraph}{4}{\z@}%
                                     {-2.0ex \@plus1ex \@minus.2ex}%
                                     {-1em}%
                                     {\normalfont\normalsize\bfseries}}
\renewcommand\subparagraph{\@startsection{subparagraph}{5}{\parindent}%
                                       {-1.5ex \@plus1ex \@minus .2ex}%
                                       {-1em}%
                                      {\normalfont\normalsize\bfseries}}
\def\@evenhead{\footnotesize\thepage\hfil\slshape\leftmark}%
\def\@oddhead{\footnotesize{\slshape\rightmark}\hfil\thepage}%
\numberwithin{equation}{section} \numberwithin{figure}{section}
\numberwithin{table}{section}

\renewcommand{\@makecaption}[2]{\begin{quote}
\footnotesize {\bf #1}~#2
\end{quote}}
\makeatother
\newcommand{\refeq}[1]{(\ref{#1})} % von htb eingefuegt

\newenvironment{summary}{\vskip\baselineskip \noindent\small\bf Summary: \rm}%
{\vskip\baselineskip}
\newenvironment{proof}{{\vskip\baselineskip\noindent\textbf{Proof:}}}%
{\hspace*{.1pt}\hspace*{\fill}\BOX\vskip\baselineskip}

\newcommand{\BOX}{\ensuremath\Box}
\newtheorem{theorem}{Theorem }[section]
\newtheorem{assumption}[theorem]{Assumption}

{\theorembodyfont{\rmfamily}}
{\theorembodyfont{\rmfamily}}
{\theorembodyfont{\rmfamily}}
\newtheorem{lemma}[theorem]{Lemma}
\newtheorem{proposition}[theorem]{Proposition}
{\theorembodyfont{\rmfamily}}
{\theorembodyfont{\rmfamily}}
{\vskip\baselineskip\noindent\textbf{Proof of {#1}:}}%
{\hspace*{.1pt}\hspace*{\fill}\BOX\vskip\baselineskip}
{\vskip\baselineskip\noindent\textbf{Proof of Theorem \protect\ref{#1}:}}%
{\hspace*{.1pt}\hspace*{\fill}\BOX\vskip\baselineskip}
{\vskip\baselineskip\noindent\textbf{Proof of Theorems \protect\ref{#1} --
\protect\ref{#2}:}}%
{\hspace*{.1pt}\hspace*{\fill}\BOX\vskip\baselineskip}
\renewcommand{\SDpagebegin}{1}          % do not change
\renewcommand{\SDpageend}{\kern1pt{}4}  % do not change
\title{Estimating the error distribution function
in nonparametric regression
\footnote{Technical report, 2004 \newline
{\em AMS 2000 subject classification:} Primary 62G05, 62G08, 62G20 \newline
{\em Key words and phrases:} Local polynomial smoother, kernel estimator,
under-smoothing, plug-in estimator, error variance,
empirical likelihood, adaptive estimator,
efficient estimator, influence function}}
\author{Ursula U. M\"uller,
Anton Schick,
Wolfgang Wefelmeyer}

\def\n{\noindent}
\def\b{\bigskip}
\def\tf{\textbf}
\def\ben{\begin{eqnarray}}
\def\be*{\begin{eqnarray*}}
\def\non{\end{eqnarray}}
\def\no*{\end{eqnarray*}}
\newcommand{\bel}[1]{\begin{equation}\label{#1}}
\newcommand{\ee}{\end{equation}}

\newcommand{\e}{\varepsilon}
\newcommand{\he}{\hat \varepsilon}
\newcommand{\vt}{\vartheta}

\newcommand{\maxi}{\max_{1 \leq i \leq n}}
\newcommand{\maxj}{\max_{1 \leq j \leq n}}
\newcommand{\avi}{\frac{1}{n} \sum_{i=1}^n}
\newcommand{\avj}{\frac{1}{n} \sum_{j=1}^n}
\newcommand{\R}{\mathbb{R}}
\newcommand{\x}{{\times}}

\newcommand{\supx}{\sup_{0\leq x \leq 1}}
\newcommand{\infx}{\inf_{0\leq x \leq 1}}
\newcommand{\supt}{\sup_{t\in \R}}
\newcommand{\suptn}{\sup_{|t|\leq n^{1/3}}}
\newcommand{\und}{\quad\textrm{and}\quad}
\def\und{\quad \mbox{and} \quad}
\def\1{\mathbf{1}}
\newcommand{\ind}[1]{\1{\{ #1 \}}}

\def\hr{\hat r}

\def\F{\mathbb{F}}
\def\hF{\hat\F}
\def\hFs{\hF_*}
\def\bFs{\bar{\F}_*}

\begin{document}
\maketitle
\thispagestyle{empty}

\begin{summary}
We construct an efficient estimator for the error distribution
function of the nonparametric regression model $Y = r(Z) + \e$.
Our estimator is a kernel smoothed empirical distribution function based on
residuals from an under-smoothed local quadratic smoother for
the regression function.
\end{summary}

\section{Introduction}

Consider the nonparametric regression model
$Y = r(Z) + \e$, where the covariate $Z$ and the error $\e$
are independent, and $\e$ has mean zero,
finite variance $\sigma^2$ and density $f$.
We observe independent copies $(Y_1,Z_1),\dots,(Y_n,Z_n)$ of $(Y,Z)$
and want to estimate the distribution  function $F$ of $\e$.
If the regression function $r$ were known, we could use the
empirical distribution function $\F$ based on the errors
$\e_1,\dots,\e_n$, defined by
\[
\F(t) = \avi \ind{\e_i \leq t}.
\]
We consider the regression function as unknown and
propose a kernel smoothed empirical distribution function $\hFs$
based on residuals from an under-smoothed local quadratic smoother for
the regression function.
We give conditions under which $\hFs$ is asymptotically equivalent to $\F$
plus some correction term:
\ben
\label{A}
\supt n^{1/2} \Big| \hFs(t) - \F(t) - f(t) \avi \e_i \Big|
= o_p(1).
\non
Smoothing the empirical distribution function is appropriate because
we assume that the error distribution has a Lipschitz density
and therefore a smooth distribution function.
A local quadratic smoother for the regression function is appropriate
because we assume that the regression function is twice
continuously differentiable.

It follows from \refeq{A} that $\hFs(t)$ has influence function
\[
\ind{\e \leq t} - F(t) + f(t)\e.
\]
M\"uller, Schick and Wefelmeyer (2004a) show that this is the
efficient influence function for estimators of $F(t)$.
Hence $\hFs$ is efficient
for $F$ in the sense that $(\hFs(t_1),\dots,\hFs(t_k))$
is a least dispersed regular estimator of $(F(t_1),\dots,F(t_k))$
for all $t_1 < \dots < t_k$ and all $k$.
The influence function of our estimator coincides with the efficient 
influence function in the model with \emph{constant} regression function;
see Bickel, Klaassen, Ritov and Wellner (1998, Section 5.5, Example 1).

It follows in particular from \refeq{A} that $\hFs(t)$ has
asymptotic variance
\[
F(t)(1-F(t))+ \sigma^2 f^2(t)  - 2 f(t) \int_{t}^{\infty} xf(x)\,dx .
\]
If  $f$ is a normal density, this simplifies to
\[
F(t)(1-F(t)) - \sigma^2  f^2(t).
\]
Hence, for normal errors, the asymptotic variance of
$\hFs(t)$ is strictly smaller than the asymptotic variance
$F(t)(1-F(t))$ of the empirical estimator $\F(t)$
based on the \emph{true} errors. This paradox is explained by the fact that
the empirical estimator $\F(t)$ is not efficient:
Unlike $\hFs(t)$, it does not
make use of the information that the errors have mean zero.
The efficient influence function for estimators of $F(t)$
from mean zero observations $\e_1,\dots,\e_n$ is
\[
\ind{\e \leq t} - F(t) - C_0(t)\e
\quad \mbox{with} \quad
C_0(t) = \sigma^{-2} \int_{-\infty}^t x f(x)\,dx;
\]
see Levit (1975).
Efficient estimators for $F(t)$ from observations
$\e_1,\dots,\e_n$ are
\[
\F(t) - \hat C_0(t) \avi \e_i
\quad \mbox{with} \quad
\hat C_0(t) = \frac{\sum_{i=1}^n \e_i \ind{\e_i \leq t}}
{\sum_{i=1}^n \e_i^2},
\]
and the empirical likelihood estimator
\[
\avi p_i \ind{\e_i \leq t}
\]
with (random) probabilities $p_i$ maximizing $\prod_{i=1}^n p_i$
subject to $\sum_{i=1}^n p_i \e_i = 0$. The empirical likelihood was
introduced by Owen (1988), (1990); see also Owen (2001).
The asymptotic variance of an efficient estimator $\F_0(t)$
for $F(t)$ from $\e_1,\dots,\e_n$ is
\[
F(t)(1-F(t)) - \sigma^{-2} \Big( \int_t^{\infty} x f(x)\,dx \Big)^2.
\]
The variance increase of our estimator $\hFs(t)$ 
over $\F_0(t)$ is therefore
\[
\Big( \sigma f(t) - \sigma^{-1} \int_t^{\infty} x f(x)\,dx \Big)^2.
\]
This is the price for not knowing the regression function.
For normal errors this term is zero, and we lose nothing.
We refer also to the introduction of M\"uller, Schick and Wefelmeyer (2004b).

Our proof is complicated by two features of the model: 
the error distribution cannot be estimated adaptively
with respect to the regression function, 
and the regression function cannot be estimated at the efficient rate
$n^{-1/2}$. Akritas and Van Keilegom (2001) encountered these problems
in a related model, the heteroscedastic regression model $Y=r(Z)+s(Z)\e$.
They used different techniques and stronger assumptions to get an expansion 
similar to \refeq{A}. Their results do not cover ours in our simpler model.

Previous related results are easier because at least one
of these complicating features is missing. Loynes (1980)
assumes that $Y = h(Z,\vt)$.
Koul (1969), (1970), (1987), (1992), Shorack (1984),
Shorack and Wellner (1986, Section 4.6) and Bai (1996)
consider linear models $Y = \vt^\top Z  + \sigma\e$.
Mammen (1996) studies the linear model as the dimension of $\vt$
increases with $n$. Klaassen and Putter (1997) and (2001)
construct efficient estimators for the error distribution function
in the linear regression model $Y = \vt^\top Z  + \e$.
Koshevnik (1996) treats the nonparametric regression model
$Y =r(Z) + \e$ with error density symmetric about zero; 
an  efficient estimator for $F$ is obtained by symmetrizing
the  empirical distribution function based on residuals.
Related results exist for time series.
See Boldin (1982), Koul (2002, Chapter 7) and Koul and Leventhal (1989)
for linear autoregressive processes $Y_j = \vt Y_{j-1} + \e_j$;
Kreiss (1991) and Schick and Wefelmeyer (2002b) for invertible linear
processes $Y_j = \e_j + \sum_{k=1}^\infty \alpha_k(\vt) \e_{j-k}$;
and Koul (2002, Chapter 8), Schick and Wefelmeyer (2002a)
and M\"uller, Schick and Wefelmeyer (2004c, Section 4) for nonlinear 
autoregressive processes $Y_j = r(\vt,Y_{j-1}) + \e_j$.
For invertible linear processes, Schick and Wefelmeyer (2004)
show that the smoothed residual-based empirical estimator is asymptotically
equivalent to the empirical estimator based on the true innovations.
General considerations on empirical processes based on estimated
observations are in Ghoudi and R\'emillard (1998).

Our result gives efficient estimators $\int h(t)\,d\hFs(t)$
for linear functionals $E[h(\e)]$ with bounded $h$.
For smooth and $F$-square-integrable functions $h$,
it is easier to prove an i.i.d.\ representation analogous to \refeq{A}
directly; see M\"uller, Schick and Wefelmeyer (2004a),
who  also use an under-smoothed estimator for the regression function.
M\"uller, Schick and Wefelmeyer (2004b) compare these results with estimation
in the larger model in which one assumes $E(\e|Z)=0$ rather than
independence of $\e$ and $Z$ with $E[\e]=0$.
A particularly simple special case is the error variance $\sigma^2$,
with $h(x) = x^2$. For the estimator $\avi \hat\e_i^2$
based on residuals $\hat\e_i = Y_i - \hr(Z_i)$
with kernel estimator $\hr$, under-smoothing
is not needed. The asymptotic variance of this  estimator was already
obtained in Hall and Marron (1990).
M\"uller, Schick and Wefelmeyer (2003) show that a
covariate-matched U-statistic is efficient for $\sigma^2$;
it does not require estimating $r$ but uses a kernel density
estimator for the covariate density $g$. There is a large literature
on simpler, inefficient, difference-based estimators for $\sigma^2$;
reviews are Carter and Eagleson (1992) and
Dette, Munk and Wagner (1998) and (1999).

We can write
\[
F(t) = \int \ind{y - r(z) \leq t} Q(dy,dz),
\]
where $Q$ is the distribution of $(Y,Z)$. Our estimator is obtained
by plugging in estimators for $Q$  and $r$. For $Q$  we use essentially
the empirical distribution; for $r$ we use a local quadratic smoother
that is under-smoothed and hence does not have the optimal rate for
estimating $r$. This means that our estimator does not obey the plug-in
principle of Bickel and Ritov (2000) and (2003).

The paper is organized as follows. Section 2 introduces our estimator and states,
in Theorem \ref{thm:main}, the assumptions needed for expansion \refeq{A}.
Section 3 derives some consequences of exponential inequalities,
and Section 4 contains properties of local polynomial smoothers.
Section 5 gives the proof of Proposition \ref{prop:2}.

\section{The estimator and the main result}

Let us now define our estimator.
We begin be defining the residuals.
This requires an estimator $\hr$ of the regression function.
We take $\hr$  to be a local quadratic smoother. To define it
we need a kernel $w$ and a bandwidth $c_n$.
A local quadratic smoother $\hr$ of $r$ is
defined as $\hr(x) = \beta_0(x)$ for $x\in [0,1]$,
where $\beta(x) = (\beta_0(x),\beta_1(x),\beta_2(x))^\top$ is the minimizer of
\[
\sum_{j=1}^n \Big( Y_j - \beta_0 - \beta_1(Z_j-x) - \beta_2(Z_j-x)^2 \Big)^2
\frac{1}{c_n} w \Big( \frac{Z_j-x}{c_n} \Big).
\]
The residuals of the regression estimator $\hr$ are
\[
\he_i = Y_i - \hr(Z_i), \quad i=1,\dots,n.
\]
Let $\hF$ denote the empirical distribution function
based on these residuals:
\[
\hF(t) = \avi \ind{\hat \e_i \leq t}, \quad t \in \R.
\]
Our estimator of the error distribution function will be a
smoothed version of $\hF$.
To this end, let $k$ be a density and $a_n$ another bandwidth.
Then we define our estimator $\hFs$ of $F$ by
\[
\hFs(t) = \int \hF(t-a_nx) k(x)\,dx,
\quad t\in \R.
\]
With $K$ the distribution function of $k$, we can write
\[
\hFs(t) = \int K \Big( \frac{t-x}{a_n} \Big)\,d\hF(x)
= \avi K \Big( \frac{t-\hat \e_i}{a_n} \Big),
\quad t\in \R.
\]
This shows that $\hFs$ is the convolution of the empirical
distribution function $\hF$ of the residuals
with the distribution function $t\mapsto K(t/a_n)$.
Alternatively,
$\hFs$ is the distribution function with density $f_*$ given by
\[
f_*(t) = \frac{1}{na_n} \sum_{i=1}^n k\Big( \frac{t-\hat \e_i}{a_n} \Big),
\quad t\in \R.
\]
This is the usual kernel density estimator of $f$ based on the residuals,
with kernel $k$ and bandwidth $a_n$.
We make the following assumptions.

\begin{assumption}\label{ass:G}
The covariate density $g$ is bounded and bounded away from zero on $[0,1]$,
and its restriction to $[0,1]$ is (uniformly) continuous.
\end{assumption}

\begin{assumption}\label{ass:R}
The regression function $r$ is twice continuously differentiable.
\end{assumption}

\begin{assumption}\label{ass:F}
The error density $f$ is Lipschitz, has mean zero, and satisfies
the moment condition $\int |x|^\gamma f(x)\,dx < \infty$ for some $\gamma >4$.
\end{assumption}

\begin{assumption}\label{ass:K}
The density $k$ is symmetric, twice continuously differentiable,
and has compact support $[-1,1]$.
\end{assumption}

\begin{assumption}\label{ass:W}
The kernel $w$ used to define the local quadratic smoother
is  a symmetric density which has compact support $[-1,1]$
and a bounded derivative $w'$.
\end{assumption}

\begin{assumption}\label{ass:BW}
The bandwidths satisfy
$a_n \sim n^{-1/4}/\log n$ and $c_n \sim n^{-1/4}$.
\end{assumption}

Note that $c_n$ is smaller than the optimal bandwidth
under Assumptions \ref{ass:G} and \ref{ass:R}.
Such a bandwidth would be proportional to $n^{-1/5}$.
This means that our choice of bandwidth results in an under-smoothed
local quadratic smoother.

We are now ready to state our main result.

\begin{theorem}\label{thm:main}
Suppose that Assumptions \ref{ass:G} to \ref{ass:BW} hold. Then
\[
\supt n^{1/2} \Big| \hFs(t) - \F(t) - f(t) \avi \e_i \Big| = o_p(1).
\]
In particular, $n^{1/2}(\hFs-F)$ converges in distribution
in the space $D([-\infty,\infty])$ to a centered Gaussian process.
\end{theorem}

\begin{proof}
For $a \in \R$ and $t \in \R$ set
\[
F_a(t) = \int F(t-ax)k(x)\,dx \und
\F_a(t) = \int \F(t-ax)k(x)\,dx.
\]
Since the density $k$ has mean zero by Assumption \ref{ass:K},
we have
\be*
F_a(t) - F(t) & = & \int \big( F(t-ax) - F(t) + ax f(t)\big) k(x)\,dx \\
 & = & \int (-ax) \int_0^1 \big( f(t-axy) - f(t) \big)\,dy\,k(x)\,dx.
\no*
Thus the Lipschitz continuity of $f$ yields
\[
\supt \big| F_{a_n}(t)-F(t) \big| = O(a_n^2) = o(n^{-1/2}).
\]
It follows from standard empirical process theory that
\bel{gn}
G_n = n^{1/2} \sup_{x\in \R} |\F_{a_n}(x)-F_{a_n}(x) - \F(x)+F(x)| = o_p(1),
\quad a_n \to 0.
\ee
Indeed, with $W_n = n^{1/2}(\F-F)$ we have
\[
G_n = \supt \Big|\int (W_n(t-a_ns)-W_n(t)) k(s)\,ds \Big|
\leq \supt \sup_{|s|\leq |a_n|} |W_n(t+s)-W_n(t)|.
\]
The above shows that
\[
\supt n^{1/2} \big| \F_{a_n}(t)-\F(t) \big| = o(1).
\]
Hence the desired result follows from Proposition \ref{prop:2} below.
\end{proof}

\begin{proposition}\label{prop:2}
Suppose that Assumptions \ref{ass:G} to \ref{ass:BW} hold.
Then
\[
\supt n^{1/2} \Big| \hFs(t) - \F_{a_n}(t) - f(t) \avi \e_i \Big| = o_p(1).
\]
\end{proposition}

The proof of Proposition \ref{prop:2} is in Section 5.
We conclude this section with a simple lemma that will
be needed repeatedly in the sequel.

\begin{lemma}\label{lm:3m}
Suppose that $\int |x|^{\beta}\,dF(x)<\infty$ for some $\beta>1$.
Then
\[
\maxi |\e_i| = o_p(n^{1/\beta}).
\]
If $F$ has also mean zero, then, as $A \to \infty$,
\[
E[\e \ind{|\e| \leq A}] = o(A^{1-\beta}).
\]
\end{lemma}

\begin{proof}
The first conclusion follows by the sharper version
of the Markov inequality: For $a>0$,
\[
P \Big( \maxi |\e_i| > a n^{1/\beta} \Big)
\leq \sum_{i=1}^n P(|\e_i| > a n^{1/\beta})
\leq a^{-\beta} E[|\e|^\beta \ind{|\e| > a n^{1/\beta}}] \to 0.
\]
The second conclusion follows from
\[
|E[\e \ind{|\e| \leq A}]| = |E[\e \ind{|\e| > A} ]|
\leq  A^{1-\beta} E[|\e|^\beta \ind{|\e| > A}] = o(A^{1-\beta}).
\]
In the first equality, we have used that $\e$ has mean zero.
\end{proof}

\section{Auxiliary Results}

In this section we derive some results
that will be used in the proof of Proposition \ref{prop:2}.
Let $(S,\mathfrak{S},Q)$ be a probability space.
For each positive integer $n$ let $V,V_1,\dots,V_n$
be independent $S$-valued random variables with
distribution $Q$, and for each $x$ in $\R$,
let $h_{nx}$ be a bounded measurable function from $S$ into $\R$.
We first study the process $H_n$ defined by
\[
H_n(x) = \avj h_{nx}(V_j)- E[h_{nx}(V)], \quad x \in \R.
\]

\begin{lemma}\label{lm:aux.1}
Let $B_n$ be a sequence of positive numbers such that
$B_n = O(n^{\alpha})$ for some $\alpha >0$. Suppose that
\bel{a.1}
\sup_{|x|\leq B_n} \Big( E[h^2_{nx}(V)]
+ \|h_{nx}\|_{\infty} \Big)
= O(n/\log n)
\ee
and, for positive numbers $\kappa_1$ and $\kappa_2$,
\bel{a.2}
\|h_{ny}-h_{nx}\|_{\infty} \leq |y-x|^{\kappa_1} O(n^{\kappa_2}),
\quad |x|,|y|\leq B_n, \quad |y-x| \leq 1.
\ee
Then
\bel{h1}
\sup_{|x|\leq B_n} |H_n(x)| = O_p(1).
\ee
If we strengthen \refeq{a.1} to
\bel{a.1a}
\sup_{|x|\leq B_n} \Big( E[h^2_{nx}(V)]
+ \|h_{nx}\|_{\infty} \Big)
= o(n/\log n),
\ee
then
\bel{h2}
\sup_{|x|\leq B_n} |H_n(x)| = o_p(1).
\ee
\end{lemma}

\begin{proof}
To prove the lemma we use an inequality of Hoeffding (1963):
If $\xi_1,\dots,\xi_n$ are independent random variables
that have mean zero and variance $\sigma^2$
and are bounded by $M$, then for $\eta > 0$,
\[
P \Big( \Big| \avj \xi_j \Big| \geq \eta \Big)
\leq 2 \exp \Big( -\frac{ n\eta^2}
{2\sigma^2 + (2/3) M \eta} \Big).
\]
Applying this inequality with $\xi_j = h_{nx}(V_j)- E[h_{nx}(V)]$,
we obtain for $\eta > 0$:
\[
P( |H_n(x)| \geq \eta)
\leq 2 \exp \Big( - \frac{n\eta^2}{2 E[h^2_{nx}(V)]
+ 2 \eta \| h_{nx} \|_\infty} \Big).
\]
Thus there is a positive number $a$
such that for all $\eta > 0$,
\[
\sup_{|x| \leq B_n} P( |H_n(x)|\geq \eta)
\leq 2 \exp \Big( - \frac{\eta^2}{1\vee \eta} a \log n \Big).
\]
Now let $x_{nk} = - B_n + 2kB_n n^{-m}$ for $k = 0,1,\dots,n^m$,
with $m$ an integer greater than $\alpha + \kappa_2/\kappa_1$.
The above yields for large enough $ \eta > 0$,
\[
P \Big( \max_{k=0,\dots,n^m}|H_n(x_{nk})| > \eta \Big)
\leq \sum_{k=0}^{n^m} P(|H_n(x_{nk})| > \eta) = o(1).
\]
Now, using \refeq{a.2},
\be*
\sup_{|x|\leq B_n} |H_n(x)|
 & \leq & \max_{k=0,\dots,n^m} \Big( |H_n(x_{nk})|
+ \sup_{|x-x_{nk}|\leq B_n n^{-m}} |H_n(x)-H_n(x_{nk})| \Big) \\
 & = &  O_p(1) + O(B_n^{\kappa_1} n^{-m \kappa_1} n^{\kappa_2}) = O_p(1).
\no*
This is the desired result \refeq{h1}.
The second conclusion is an immediate consequence.
\end{proof}

Next we consider the degenerate U-process
\[
U_n(x) =  n^{-m/2} \sum_{(i_1,\dots,i_m)\in I_m^n}
u_{nx}(V_{i_1},\dots,V_{i_m}), \quad x \in \R,
\]
with $I_m^n = \{ (i_1,\dots,i_m): 1 \leq i_j \leq n, i_j \neq i_k
\mbox{ if } j \neq k\}$,
and $u_{nx}$ a bounded measurable function from $S^m$ to $\R$
such that for all $v_1,\dots,v_m$ in $S$,
\[
E[u_{nx}(V_1,v_2,\dots,v_m)] =
\dots = E[u_{nx}(v_1,v_2,\dots,V_m)] = 0.
\]
Set $\|u_{nx}\|_2 = (E [u_{nx}^2(V_1,\dots,V_m)])^{1/2}$.

\begin{lemma}\label{lm:aux.2}
Let $B_n$ be positive numbers such that $B_n = O(n^\alpha)$
for some $\alpha >0 $. Suppose that
\bel{u.1}
\sup_{|x|\leq B_n} \Big( \|u_{nx}\|_2^{2/m}
+ \| u_{nx}\|_{\infty}^{2/(m+1)} n^{-1/(m+1)} \Big) = O((\log n)^{-1})
\ee
and, for some positive $\kappa_1$ and $\kappa_2$,
\bel{u.2}
\| u_{ny}-u_{nx}\|_{\infty} \leq |y-x|^{\kappa_1} O(n^{\kappa_2}),
\quad |x|,|y| \leq B_n, \quad |y-x| \leq 1.
\ee
Then
\bel{U1}
\sup_{|x|\leq B_n} |U_n(x)| = O_p(1).
\ee
If we strengthen \refeq{u.1} to
\bel{u.1a}
\sup_{|x|\leq B_n} \Big( \|u_{nx}\|_2^{2/m}
+ \| u_{nx}\|_{\infty}^{2/(m+1)} n^{-1/(m+1)} \Big) = o((\log n)^{-1}),
\ee
then
\bel{U2}
\sup_{|x|\leq B_n} |U_n(x)| = o_p(1).
\ee
\end{lemma}

\begin{proof}
We use a similar argument as for Lemma \ref{lm:aux.1}, but rely now on the
Arcones--Gin\'e exponential inequality for degenerate U-processes
(inequality (c) in Proposition 2.3 of Arcones and Gin\'e, 1994).
This inequality states that there
are constants $c_1$ and $c_2$ depending only on $m$ such that,
for every $\eta>0$, all $x$ and all $n$,
\[
P(|U_{n}(x)|>\eta)
\leq c_1 \exp \Big( - \frac{c_2 \eta^{2/m}}{\|u_{nx}\|_2^{2/m}
+ (\|u_{nx}\|_\infty \eta^{1/m} n^{-1/2})^{2/(m+1)}} \Big).
\]
{From} this inequality one obtains as in the proof of Lemma \ref{lm:aux.1} that
there is a positive number $b$ such that
\[
\sup_{|x|\leq B_n} P(|U_n(x)| > \eta)
\leq c_1 \exp \Big( - \frac{\eta^{2/m}}
{(1 \vee \eta)^{2/(m+m^2)}} b \log n \Big), \quad \eta > 0.
\]
Now proceed as in the proof of Lemma \ref{lm:aux.1}.
\end{proof}

\section{Properties of local polynomial smoothers}

For an introduction to local polynomial smoothers we refer to
Fan and Gijbels (1996).
In this section we derive some properties of local polynomial
smoothers $\hr$ of order $d$, defined by
$\hr(x) = \beta_0(x)$ for $x \in [0,1]$, where
$\beta(x) = (\beta_0(x),\dots,\beta_d(x))^\top$ is the minimizer of
\[
\sum_{j=1}^n \Big( Y_j - \sum_{m=0}^d \beta_m
\Big( \frac{Z_j-x}{c_n} \Big)^m \Big)^2
\frac{1}{c_n} w \Big( \frac{Z_j-x}{c_n} \Big).
\]
Here we have re-scaled $\beta_1,\dots,\beta_d$ for convenience.
The normal equations are
\[
Q_n(x) \beta = \avj w_n(Z_j-x) Y_j,
\]
where the vector $w_n(x) = (w_{n0}(x),\dots,w_{nd}(x))^\top$
has entries
\[
w_{nm}(x) = \frac{x^m}{c_n^{m+1}} \, w\Big( \frac{x}{c_n} \Big),
\]
and the matrix $Q_n(x)$ has entries $q_{n,k+m}(x)$, $k,m=0,\dots,d$,
with
\[
q_{nm}(x) = \avj w_{nm}(Z_j-x).
\]
By the properties of the kernel $w$ and the covariate density $g$ 
we have for $m=0,\dots,2d$ and all $x \in \R$,
\ben
\label{w1}
 |w_{nm}(x)| & \leq & \|w\|_{\infty} c_n^{-1}, \\
\label{w2}
|w'_{nm}(x)| & \leq &  (\|w'\|_{\infty} + m \|w\|_{\infty}) c_n^{-2}, \\
\label{w3}
E[w_{nm}^2(Z-x)] & \leq & \|w\|_{\infty} \|g\|_{\infty} c_n^{-1}.
\non
Write $p_n(x) = (p_{n0}(x),\dots,p_{nd}(x))^\top$ for the first column
of the inverse $Q_n(x)^{-1}$ of $Q_n(x)$, and
\[
A_n(x,y) = p_n(x)^\top w_n(y-x).
\]
{From} the normal equations we obtain
\[
\hr(x) = \beta_0(x) = \avj A_n(x,Z_j) Y_j.
\]
For the expectation of $q_{nm}(x)$ we write
\[
\overline q_{nm}(x) = E[q_{nm}(x)] = \int g(x+c_n t) t^m w(t)\,dt.
\]
We define $\bar Q_n(x)$ correspondingly, replacing $q_{nm}(x)$
by $\overline q_{nm}(x)$.
Furthermore, $\overline p_n$ and $\bar A_n$
are defined as $p_n$ and $A_n$, with $Q_n$ replaced by $\bar Q_n$.

For a unit vector $v=(v_0,\dots,v_d)^\top$ and $0\leq x\leq 1$ we have
\[
v^\top\bar Q_n(x)v = \int \Big( \sum_{i=0}^d v_i t^i \Big)^2 g(x+c_n t)w(t)\,dt.
\]
Thus, by Assumption \ref{ass:G},
\[
(d+1) \|g\|_\infty \geq v^\top\bar Q_n(x)v
\geq \infx g(x) \int_{-1\vee(-x/c_n)}^{1\wedge((1-x)/c_n)}
\Big( \sum_{i=0}^d v_i t^i \Big)^2 w(t)\,dt.
\]
By Assumption \ref{ass:W}, there is an $\eta>0$ such that the eigenvalues
of $\bar Q_n(x)$ are in the interval $[\eta,(d+1)\|g\|_\infty]$ 
for all $x$ and $n$.
Thus $\bar Q_n(x)$ is invertible, and
\bel{star3}
\supx \|\bar Q_n^{-1}(x)\| \leq 1/\eta.
\ee

\begin{lemma}\label{lm:l1}
Suppose Assumptions \ref{ass:G} and \ref{ass:W} hold.
Let $c_n\to 0$ and $c_n^{-1}=O(n/\log n)$.
Then
\[
\supx \Big| q_{nm}(x) - \overline q_{nm}(x) \Big|
= O_p((nc_n/\log n)^{-1/2}), \quad m=0,\dots,2d,
\]
and consequently
\[
\supx \| Q_n(x)-\bar Q_n(x) \| = O_p((nc_n/\log n)^{-1/2}).
\]
\end{lemma}

\begin{proof}
Fix $m$ and use Lemma \ref{lm:aux.1} with $B_n=1$, $V_j=Z_j$ and
\[
h_{nx}(v)= (nc_n/\log n)^{1/2} w_{nm}(v-x).
\]
For these choices,
the conditions \refeq{a.1} and \refeq{a.2}, with $\kappa_1 = \kappa_2 = 1$,
follow from \refeq{w1} to \refeq{w3}.
\end{proof}

\begin{lemma}\label{lm:l2}
Suppose Assumptions \ref{ass:G} and \ref{ass:W} hold.
Assume also that $f$ has mean zero and finite moment of order $\beta>2$.
Let $c_n\to 0$ and $c_n^{-1}n^{2/\beta}=O(n/\log n)$.
Then
\[
\supx \Big| \avj w_{nm}(Z_j-x) \e_{j} \Big|
= O_p((nc_n/\log n)^{-1/2}), \quad m=0,\dots,2d.
\]
\end{lemma}

\begin{proof}
Fix $m$. In view of Lemmas \ref{lm:3m} and \ref{lm:l1}
it suffices to show that
\bel{star}
\supx \Big| \avj w_{nm}(Z_j-x)\e_{nj} \Big|
= O_p((nc_n/\log n)^{-1/2}),
\ee
where
$\e_{nj} = \e_j \ind{|\e_j| \leq n^{1/\beta}} - E[\e \ind{|\e| \leq n^{1/\beta}}]$.
Here we used the fact that
\[
(nc_n/\log n)^{1/2}E[\e\ind{|\e|\leq n^{1/\beta}}]
= O(n^{-1/2}c_n^{1/2}n^{1/\beta}) = o(1).
\]
But \refeq{star} follows from an application of Lemma \ref{lm:aux.1}
with $B_n=1$, $V_j=(Z_j,\e_j)$ and
\[
h_{nx}(Z_j,\e_j) = (nc_n/\log n)^{1/2} w_{nm}(Z_j-x) \e_{nj}.
\]
Indeed, the left-hand side of \refeq{a.1} is of order
$n/\log n +(nc_n/\log n)^{1/2}c_n^{-1}n^{1/\beta}$,
which is of order $n/\log n$ by the assumptions on $c_n$.
Relation \refeq{a.2} follows by the Lipschitz continuity of $w$.
\end{proof}

\begin{theorem}\label{thm:lls}
Suppose Assumptions \ref{ass:G} and \ref{ass:W} hold.
Assume also that $f$ has mean zero and finite moment of order $\beta>2$.
Let $c_n\to 0$ and $c_n^{-1}n^{2/\beta}=O(n/\log n)$.
Then
\bel{ae}
\supx \Big| \avj \bar A_n(x,Z_j)\e_j \Big|
= O_p((nc_n/ \log n)^{-1/2}) .
\ee
If, in addition,
$r$ is $\nu$-times continuously differentiable with $\nu\leq d$,
then
\bel{er}
\supx \Big| \hr(x) - r(x) - \avj \bar A_{n}(x,Z_j)\e_j \Big|
= O_p(\log n/(nc_n)) + o_p(c_n^{\nu}).
\ee
If $r$ has a Lipschitz continuous $d$-th derivative, then $o_p(c_n^{\nu})$
can be replaced by $O_p(c_n^{d+1})$.
\end{theorem}

\begin{proof}
Since
\[
\avj \bar A_n(x,Z_j) \e_j = \overline p_n(x)^\top \avj w_n(Z_j-x) \e_j,
\]
relation \refeq{ae} follows from \refeq{star3} and Lemma \ref{lm:l2}.
To prove \refeq{er}, write
\[
\hr(x) = \tilde r(x) + p_n(x)^\top \avj w_n(Z_j-x) \e_j
\]
with
\[
\tilde r(x) = \avj A_n(x,Z_j) r(Z_j).
\]
By Lemma \ref{lm:l1} and relation \refeq{star3},
\bel{cd}
\supx \|p_n(x) - \overline p_n(x)\| = O_p((nc_n/\log n)^{-1/2}).
\ee
In view of this and Lemma \ref{lm:l2}, assertion \refeq{er} follows
if we verify
\bel{cr}
\supx |\tilde r(x)-r(x)| = o_p(c_n^{\nu}).
\ee
By construction,
\[
\sum_{j=1}^n A_{n}(x,Z_j) = 1
\und
\sum_{j=1}^n A_{n}(x,Z_j) (x-Z_j)^{m} = 0, \quad m=1,\dots,d.
\]
Hence, if we assume that $r$ is $\nu$-times continuously differentiable with
$\nu\leq d$, we can write
\[
\tilde r(x) - r(x) = \avj A_n(x,Z_j)
\bigg( r(Z_j) - r(x) - \sum_{m=1}^{\nu} r^{(m)}(x)
 \frac{(Z_j-x)^m}{m!} \bigg)
\]
and obtain the bound
\[
|\tilde r(x) - r(x)| \leq
\avj |A_n(x,Z_j)|  \; \frac{c_n^{\nu}} {\nu !}
\sup _{z\in [0,1], |z-x| \leq c_n} |r^{(\nu)} (z)- r^{(\nu)}(x)|.
\]
By \refeq{cd}, Lemma \ref{lm:l1} and \refeq{star3},
\bel{Ai}
\supx \avj |A_n(x,Z_j)| = O_p(1).
\ee
The desired \refeq{cr} follows from this and the uniform continuity of
$r^{(\nu)}$ on $[0,1]$.
If the $d$-th derivative is Lipschitz, one readily sees that \refeq{cr}
holds with $o_p(c_n^{\nu})$ replaced by $O_p(c_n^{d+1})$.
\end{proof}

We conclude this section by pointing out an additional property
of $\bar A_n$.

\begin{lemma}\label{lm:Abar}
Suppose Assumptions \ref{ass:G} and \ref{ass:W} hold.
Let $c_n\to 0$ and $c_n^{-1}=O(n/\log n)$.
Then
\bel{Abar2}
\avi \Big(\avj \bar A_{n}(Z_j,Z_i)-1\Big)^2 = o_p(1).
\ee
\end{lemma}

\begin{proof}
Since
\[
|\bar A_n(z,x)| \leq \sup_{0\leq y \leq 1}\|\overline p_n(y)\|(d+1)^{1/2}
\frac{1}{c_n} w\Big( \frac{x-z}{c_n} \Big),
\]
we obtain from Lemma \ref{lm:l1} and the properties of $\bar Q_n$ that
\bel{Abar1}
\supx \avj  |\bar A_{n}(x,Z_j)|  +
\supx \avj  |\bar A_{n}(Z_j,x)| = O_p(1).
\ee
Let $\Sigma$ be the $(d+1)\x (d+1)$ matrix
with $(i,j)$-entry given by $\int t^{i+j-2}k(t)\,dt$.
It follows from the uniform continuity of $g$ on $[0,1]$ that
\[
\sup_{c_n <x,y < 1-c_n, |x-y|\leq c_n} \|\bar Q_n(y)- g(x) \Sigma\| = o(1).
\]
This and Lemma \ref{lm:l1} yield
\[
\sup_{c_n <x< 1-c_n}
\Big| \avj w_{nm}(x-Z_j) - g(x) \int (-t)^m k(t)\,dt \Big| = o_p(1)
\]
for $m=0,\dots,2d$.
Let $u$ denote the first column of $\Sigma$ and $v$ be the first column of
$\Sigma^{-1}$. Then we have
\[
\sup_{c_n <x,y < 1-c_n, |x-y|\leq c_n}
\Big\| \overline p_n(y) - \dfrac{1}{g(x)}v \Big\| = o(1)
\]
and
\[
\sup_{c_n <x< 1-c_n} \Big\| \avj w_{n}(x-Z_j) - g(x) u \Big\| = o(1).
\]
Since $v^{\top} u=1$, we immediately obtain that
\[
\sup_{2c_n <x< 1-2c_n} \Big| \avj \overline p_n^{\top}(Z_j)w_n(x-Z_j)-1 \Big|
= o_p(1).
\]
The desired result follows from this, \refeq{Abar1}
and the fact that, by Assumption \ref{ass:G},
\[
\avi (\ind{ Z_i < 2c_n} + \ind{Z_i > 1-2c_n})= o_p(1).
\]
\end{proof}

\section{Proof of Proposition \ref{prop:2}}

Let $q_n = (nc_n/\log n)^{-1/2}$.
By choice of $c_n$ and $a_n$ we have $q_n=o(a_n)$.
It follows from Theorem \ref{thm:lls} that our local quadratic smoother
$\hat r$ satisfies
\[
\sup_{0\leq x \leq 1} |\hat r(x)- r(x)| = O_p(q_n).
\]
This and Lemma \ref{lm:3m} yield that the probability of the event
$\{ \maxj |\e_j|\geq n^{1/3} -a_n\}\cup\{\maxj |\hat \e_j|\geq n^{1/3} -a_n\}$
tends to zero.
On the complement of this event
we have $\hFs(t)=\F_{a_n}(t) =0$ for all $t <-n^{1/3}$ and
$\hFs(t)=\F_{a_n}(t)=1$ for all $t > n^{1/3}$.
Finally, $\sup_{|t|>n^{1/3}} f(t)=o(1)$ by the uniform continuity of $f$.
Combining the above and the fact that $\avj \e_j = O_p(n^{-1/2})$, we obtain
that
\[
\sup_{|t|>n^{1/3}} \Big| \hFs(t)-\F_{a_n}(t)-f(t) \; \avj \e_j \Big|
= o_p(n^{-1/2}).
\]
Thus we need to show that
\[
\suptn \Big| \hFs(t)- \F_{a_n}(t) - f(t)\; \avj \e_j \Big| = o_p(n^{-1/2}).
\]
For this, we first derive some preparatory results.
Let $\phi$ be a Lipschitz-continuous function with compact support
contained in $[-1,1]$, and $b_n$ a sequence of positive numbers
such that $b_n\to 0$ and $b_n^{-1}=O(n/\log n)$.
Then it follows from Lemma \ref{lm:aux.1} that
\bel{fe1}
\sup_{|x|\leq n^{1/3}}
\Big| \frac{1}{nb_n} \sum_{j=1}^n \phi\Big(\frac{x-\e_j}{b_n}\Big) -
\int f(x-b_nt)\phi(t)\,dt \Big|
= O_p( (\log n/(nb_n))^{1/2}).
\ee
Since $f$ is Lipschitz,
\bel{fe2}
\sup_{x\in \R}
\Big| \int f(x-b_nt)\phi(t)\,dt - f(x)\int \phi(t)\,dt \Big| = O(b_n).
\ee
It follows from \refeq{fe1} and \refeq{fe2}, with $\phi$ replaced by $|\phi|$,
that
\bel{fe3}
\sup_{|x|\leq n^{1/3}} \frac{1}{nb_n} \sum_{j=1}^n
\Big| \phi \Big( \frac{x-\e_j}{b_n} \Big) \Big| = O_p(1).
\ee
Next, let $\psi$ be the triangular density defined by
$\psi(x)= (1-|x|) \ind{|x|\leq 1}$.
Then we have for $x\in \R$ and $u\in \R$ with $|u|\leq a_n$ that
\[
\Big| \phi \Big( \frac{x-u}{a_n} \Big) \Big| \leq
\| \phi \|_{\infty} \ind{ |x| \leq 2a_n}
\leq 2\|\phi\|_\infty \psi\Big(\frac{x}{4a_n}\Big).
\]
This shows that for all $t\in\R$ and random variables
$\xi_{n,j}$ and $\zeta_{n,j}$ we have
\[
\Big| \avj \phi\Big( \frac{t-\e_j + \xi_{n,j}}{a_n} \Big)
\zeta_{n,j} \ind{|\xi_{n,j}| \leq a_n} \Big|
\leq 2\|\phi\|_\infty \avj \psi\Big( \frac{t-\e_j}{4a_n} \Big) |\zeta_{n,j}|.
\]
Thus if $\max_{1\leq j\leq n}|\xi_{n,j}| = o_p(a_n)$, we have
\bel{phipsi}
\suptn \Big| \dfrac{1}{na_n} \sum_{j=1}^n  
\phi\Big( \frac{t-\e_j + \xi_{n,j}}{a_n} \Big) \zeta_{n,j} \Big|
= O_p \Big( \max_{1\leq j\leq n}|\zeta_{n,j}| \Big).
\ee

%Now let $K_n(t)= K(t/a_n)$ for $t\in \R$. 
We can write
\[
%\hFs(t) = \avj K_n(t-\hat \e_j)
%= \avj K_n(t-\e_j + \hr(Z_j) - r(Z_j)), \quad t\in \R.
\hFs(t) = \avj K\Big(\frac{t-\hat \e_j}{a_n}\Big)
= \avj K\Big(\frac{t-\e_j + \hr(Z_j) - r(Z_j)}{a_n}\Big), \quad t\in \R.
\]
Choose $\gamma>4$ such that $\int|x|^\gamma f(x)\,dx<\infty$.
By Lemma \ref{lm:3m} we have
\bel{add}
P(\maxi |\e_i| > n^{1/\gamma}) \to 0
\quad\textrm{and}\quad
E[\e_1 \ind{|\e_1|\leq n^{1/\gamma}}]= o(n^{-1/2}).
\ee
Let $\e_{n,j}= \e_j \ind{|\e_j|\leq n^{1/\gamma}}-
E[\e_{j} \ind{|\e_j|\leq n^{1/\gamma}}]$.  Set
\[
\begin{aligned}
\delta_{n,i} & = \avj \bar A_n(Z_i,Z_j) \e_{n,j},
\quad i=1,\dots,n, \\
\bFs(t) & = 
%\avj K_n(t-\e_j + \delta_{n,j}), \quad  t\in \R.
\avj K\Big(\frac{t-\e_j + \delta_{n,j}}{a_n}\Big), \quad  t\in \R.
\end{aligned}
\]
Our next goal is to show that
\bel{wf1}
\suptn |\hFs(t)-\bFs(t)| = o_p(n^{-1/2}).
\ee
It follows from \refeq{star3} and the properties of $w$ that
\bel{ab}
|\bar A_{n}(x,z)| \leq C_n \ind{|x-z| \leq c_n},
\quad x,z\in [0,1],
\ee
where $C_n = O(c_n^{-1})$.
In view of (\ref{Abar1}) and (\ref{add}),
Theorem \ref{thm:lls} and Lemma \ref{lm:Abar} yield
\ben
\label{ddif}
\maxi |\hr(Z_i) - r(Z_i) - \delta_{n,i} | & = & o_p(n^{-1/2}), \\
\label{dmax}
\maxi |\delta_{n,i}| & =  & O_p(q_n), \\
\label{dave}
\avj \delta_{n,j} - \avj \e_j  & = & o_p(n^{-1/2}).
\non
With $\zeta_{n,j}=\hr(Z_j)-r(Z_j)-\delta_{n,j}$ we have
\[
\hFs(t)-\bFs(t) = \int_0^1 \dfrac{1}{na_n} \sum_{j=1}^n \zeta_{n,j}
k\Big(\frac{t-\e_j+\delta_{n,j}+s\zeta_{n,j}}{a_n}\Big)\,ds.
\]
Using \refeq{ddif}, \refeq{dmax}, and \refeq{phipsi} with $\phi=k$,
we obtain \refeq{wf1}.

A Taylor expansion shows that
\[
\bFs(t) - \F_{a_n}(t) = T_{n,1}(t) + \dfrac{1}{2}T_{n,2}(t) 
+ \dfrac{1}{2}R_n(t),
\]
where
\[
%\begin{aligned}
T_{n,1}(t) %& = \avj K_n'(t-\e_j) \delta_{n,j}
= \frac{1}{na_n} \sum_{j=1}^n k\Big(\frac{t-\e_j}{a_n}\Big) \delta_{n,j}, 
\qquad
%\\
T_{n,2}(t) %& = \avj K_n''(t-\e_j) \delta_{n,j}^2
= \frac{1}{na_n^2} \sum_{j=1}^n k'\Big( \frac{t-\e_j}{a_n}\Big)\delta_{n,j}^2,
%\\
%\end{aligned}
\]
and
\[
R_n(t) = \int_0^1 \frac{1}{na_n^3}
\sum_{j=1}^n k''\Big( \frac{t-\e_j+s\delta_{n,j}}{a_n} \Big) \delta_{n,j}^3 (1-s)^2\,ds.
\]
By \refeq{phipsi} and \refeq{dmax},
\[
\suptn |R_n(t)| = O_p(a_n^{-2} q_n^3) = o_p(n^{-1/2}).
\]
For $t\in \R$, let now
\[
\begin{aligned}
S_{n,1}(t) & = \int f(t-a_nx) k(x)\,dx  \; \avj \delta_{n,j}, \\
S_{n,2}(t) & = a_n^{-1}\int f(t-a_nx) k'(x)\,dx \; \avj \delta^2_{n,j}.
\end{aligned}
\]
In view of the Lipschitz continuity of $f$, it follows from \refeq{dave} that
\[
\suptn \Big| S_{n,1}(t) - f(t) \; \avj \e_j \Big| = o_p(n^{-1/2}),
\]
and from $\int k'(x)\,dx=0$ that
\[
\suptn |S_{n,2}(t)| = O_p\Big( \maxi \delta_{n,i}^2 \Big)
= O_p(q_n^2) = o_p(n^{-1/2}).
\]
Thus the desired result will follow if we show that
\bel{wts}
\suptn |T_{n,\nu}(t)-S_{n,\nu}(t)| = o_p(n^{-1/2}), \quad \nu=1,2.
\ee
We shall demonstrate this for the case $\nu=2$. The case $\nu=1$ is similar,
yet simpler.

We can write
$n^{1/2} (T_{n,2}(t)-S_{n,2}(t))
= \sum_{i=1}^6 U_{n,i}(t)$, where
\[
\begin{aligned}
U_{n,1}(t) & =  n^{-3/2}
\sum_{(i,j,l)\in I_3^n} \phi_{n,t}(\e_{i})
n^{-1} a_n^{-2} \bar A_n(Z_{i},Z_{j}) \bar A_n(Z_{i},Z_{l})
\e_{n,j} \e_{n,l}, \\
U_{n,2}(t) & = n^{-1} \sum_{(i,j)\in I_2^n} \phi_{n,t}(\e_{i})
n^{-3/2} a_n^{-2}\bar A_n^2(Z_i,Z_j) (\e_{n,j}^2- E[\e_{n,1}^2]), \\
U_{n,3}(t) & = n^{-1} \sum_{(i,j)\in I_2^n} \phi_{n,t}(\e_{i})
n^{-3/2} a_n^{-2} 
%\\ & \phantom{n^{-1} \sum_{(i,j)\in I_2^n}}
%\Big(\bar A_n^2(Z_i,Z_j)-\int \bar A_n^2(Z_i,z) g(z)\,dz \Big) E[\e_{n,1}^2], 
(\bar A_n^2(Z_i,Z_j) - \bar A_{ni}^2) E[\e_{n,1}^2], 
\\
U_{n,4}(t) & = \avi \phi_{n,t}(\e_{i}) (n-1)n^{-3/2} a_n^{-2} 
\bar A_{ni}^2[\e_{n,1}^2], 
%\int \bar A_n^2(Z_i,z)g(z)\, dz E[\e_{n,1}^2], 
\\
U_{n,5}(t) & = \dfrac{2}{n}\sum_{i=1}^n n^{-1/2}a_n^{-2} \tilde\delta_{n,i}
\phi_{n,t}(\e_{i})\bar A_n(Z_{i},Z_{i})\e_{n,i}, \\
U_{n,6}(t) & = \avi n^{-3/2}\phi_{n,t}(\e_{i})\bar A_n^2(Z_{i},Z_{i})\e_{n,i}^2,
\end{aligned}
\]
with
\[
\begin{aligned}
\phi_{n,t}(\e_i) & = k'\Big( \frac{t-\e_i}{a_n} \Big)
- E \Big[ k' \Big( \frac{t-\e_1}{a_n} \Big) \Big], \\
\bar A_{ni}^2 & =\int \bar A_n^2(Z_i,z)g(z)\, dz , \\
\tilde\delta_{n,i} & = \frac{1}{n}\sum_{j:j\neq i}\bar A_n^2(Z_i,Z_j)\e_{n,j}.
\end{aligned}
\]
Thus we are left to show that
\bel{wu}
\suptn |U_{n,\nu}(t) | = o_p(1), \quad \nu=1,\dots,6.
\ee
For $\nu=1,2,3$ we verify \refeq{wu} with the aid of Lemma \ref{lm:aux.2}.
In each case, \refeq{u.2} is a consequence of the Lipschitz continuity of $k'$.
Thus we only check \refeq{u.1a}.

Note that $U_{n,1}(t)$ is a degenerate U-statistic of order 3.
By \refeq{ab}, its kernel $u_{n,t}$ satisfies $\supt \| u_{n,t}\|_{\infty}
= O(n^{-1} a_n^{-2} c_n^{-2} n^{2/\gamma})$ and
$\supt \|u_{n,t}\|^2_2= O(n^{-2} a_n^{-3} c_n^{-2})$.
Since $(n^{-2} a_n^{-3} c_n^{-2})^{1/3}
+ (n^{-1} a_n^{-2} c_n^{-2} n^{2/\gamma})^{1/2}n^{-1/4} = o((\log n)^{-1})$,
we have \refeq{u.1a} and hence obtain \refeq{wu} for $\nu=1$.

Note that $U_{n,2}(t)$ is a degenerate U-statistic of order 2.
Its kernel $u_{n,t}$ satisfies
$\supt \| u_{n,t}\|_{\infty} = O(n^{-3/2} a_n^{-2} c_n^{-2} n^{2/\gamma}$
and $\supt \|u_{n,t}\|^2_2=O(n^{-3} a_n^{-3} c_n^{-3})$.
Thus we have \refeq{u.1a} and hence \refeq{wu} for $\nu=2$.

Finally, $U_{n,3}(t)$ is a degenerate U-statistic of order 2.
Its kernel $u_{n,t}$ satisfies
\[
\supt \| u_{n,t}\|_{\infty} = O(n^{-3/2} a_n^{-2} c_n^{-2}).
\]
Thus we have \refeq{u.1a} and hence \refeq{wu} for $\nu=3$.

We apply Lemma \ref{lm:aux.1} to obtain \refeq{wu} with $\nu=4$.
We have \refeq{a.1a} since its left-hand side is of order
$n^{-1} a_n^{-3} c_n^{-2} + n^{-1/2} a_n^{-2} c_n^{-1}$.
Of course, \refeq{a.2} follows since $k'$ is Lipschitz.
Thus we can apply Lemma \ref{lm:aux.1} and conclude \refeq{wu}
with $\nu=4$.

We obtain from \refeq{fe2} and \refeq{fe3} with $\phi=k'$ that
\[
\suptn \avi|\phi_{n,t}(\e_i)| = O_p(a_n).
\]
Since $\maxi |A_n(Z_{i},Z_{i})\e_{n,i}| = O_p(c_n^{-1}n^{1/\gamma})$,
we obtain from \refeq{dmax} that
\[
\maxi |\tilde\delta_{n,i}| = O_p(q_n).
\]
Thus we have
\[
\begin{aligned}
\suptn |U_{n,5}(t)| & = O_p(q_n a_n^{-1} c_n^{-1} n^{1/\gamma} n^{-1/2}) =o_p(1) \\
%& = O_p(\log n)^{1/2} n^{-1} a_n^{-1} c_n^{-3/2} n^{1/\gamma}) = o_p(1), \\
\suptn |U_{n,6}(t)| & = O_p(q_n a_n^{-1} n^{-3/2} c_n^{-2} n^{2/\gamma}) 
=o_p(1).
%& = O_p((\log n)^{1/2} n^{-2} a_n^{-1} c_n^{-5/2} n^{2/\gamma}) = o_p(1).
\end{aligned}
\]

%\b\n\tf{Acknowledgments} \\
%The research of Anton Schick was supported in part by NSF Grant DMS 0072174.
%The paper was finished while two of the authors stayed at the
%Oberwolfach Institute under the program ``Research in Pairs''.
%We thank Christoph Scheicher for valuable comments.

\b\n
Ursula U. M\"uller \\
Fachbereich 3: Mathematik und Informatik \\
Universit\"at Bremen \\
Postfach 330 440 \\
28334 Bremen, Germany

\b\n
Anton Schick \\
Department of Mathematical Sciences \\
Binghamton University \\
Binghamton, New York 13902-6000, USA

\b\n
Wolfgang Wefelmeyer \\
Mathematisches Institut \\
Universit\"at zu K\"oln \\
Weyertal 86-90 \\
50931 K\"oln, Germany


\begin{thebibliography}{20}


\bibitem{av01}
Akritas, M. G. and Van Keilegom, I. (2001).
Non-parametric estimation of the residual distribution.
\emph{Scand.\ J. Statist.}\ \tf{28}, 549--567.

\bibitem{ag93}
Arcones, M. A. and Gin\'e, E. (1993).
Limit theorems for U-processes.
\emph{Ann.\ Probab.}\ \tf{21}, 1494--1542.

\bibitem{bai96}
Bai, J. (1996).
Testing for parameter constancy in linear regressions: An empirical
distribution function approach.
\emph{Econometrica}\ \tf{64}, 597--622.

\bibitem{bkrw98}
Bickel, P. J., Klaassen, C. A. J., Ritov, Y. and Wellner, J. A. (1998).
\emph{Efficient and Adaptive Estimation for Semiparametric Models}.
Springer, New York.

\bibitem{br00}
Bickel, P. J. and Ritov, Y. (2000).
Non- and semiparametric statistics: Compared and contrasted.
\emph{J. Statist.\ Plann.\ Inference}\ \tf{91}, 209--228.

\bibitem{br03}
Bickel, P. J. and Ritov, Y. (2003).
Non-parametric estimators which can be ``plugged-in''.
\emph{Ann. Statist.}\ \tf{31}, 1033--1053.

\bibitem{ce92}
Carter, C. K. and Eagleson, G. K. (1992).
A comparison of variance estimators in nonparametric regression.
\emph{J. Roy.\ Statist.\ Soc.\ Ser.\ B. (Methodological)}\
\tf{54}, 773--780.

\bibitem{dmw98}
Dette, H., Munk, A. and Wagner, T. (1998).
Estimating the variance in nonparametric regression ---
what is a reasonable choice?
\emph{J. Roy.\ Statist.\ Soc.\ Ser.\ B. (Methodological)}\
\tf{60}, 751--764.

\bibitem{dmw99}
Dette, H., Munk, A. and Wagner, T. (1999).
A review of variance estimators with extensions to multivariate
nonparametric regression models.
In: \emph{Multivariate Analysis, Design of Experiments,
and Survey Sampling} (S. Ghosh, ed.), 469--498,
Statistics: Textbooks and Monographs 159,
Dekker, New York.

\bibitem{fg96}
Fan, J. and Gijbels, I. (1996).
\emph{Local polynomial modelling and its applications}.
Monographs on Statistics and Applied Probability 66,
Chapman \& Hall, London.

\bibitem{gr98}
Ghoudi, K. and R\'emillard, B. (1998).
Empirical processes based on pseudo-observations.
In: \emph{Asymptotic Methods in Probability and Statistics}
(B. Szyszkowicz, ed.), 171--197, North-Holland, Amsterdam.

\bibitem{hm90}
Hall, P. and Marron, J. S. (1990).
On variance estimation in nonparametric regression.
\emph{Biometrika}\ \tf{77}, 415--419.

\bibitem{hoeff63}
Hoeffding, W. (1963).
Probability inequalities for sums of bounded random variables.
\emph{J. Amer.\ Statist.\ Assoc.}\ \tf{58}, 13--30.

\bibitem{kp97}
Klaassen, C. A. J. and Putter, H. (1997).
Efficient estimation of the error distribution
in a semiparametric linear model.
In: \emph{Contemporary Multivariate Analysis and Its Applications}
(K. T. Fang and F. J. Hickernell, eds.), 1--8,
Hong Kong Baptist University.

\bibitem{kp01}
Klaassen, C. A. J. and Putter, H. (2001).
Efficient estimation of Banach parameters in semiparametric models. \\
Available at: http://www.medstat.medfac.leidenuniv.nl/ms/Hp/.

\bibitem{kosh96}
Koshevnik, Yu. A. (1996).
Semiparametric estimation of a symmetric error distribution
from regression models.
\emph{Publ.\ Inst.\ Statist.\ Univ.\ Paris}\ \tf{40}, 77--91.

\bibitem{koul69}
Koul, H. L. (1969).
Asymptotic behavior of Wilcoxon type confidence regions in multiple linear
regression.
\emph{Ann.\ Math.\ Statist.}\ \tf{40}, 1950--1979.

\bibitem{koul70}
Koul, H. L. (1970).
Some convergence theorems for ranks and weighted empirical cumulatives.
\emph{Ann.\ Math.\ Statist.}\ \tf{41}, 1273--1281.

\bibitem{koul87}
Koul, H. L. (1987).
Tests of goodness-of-fit in linear regression.
\emph{Colloq.\ Math.\ Soc.\ J\'anos Bolyai}\ \tf{45}, 279--315.

\bibitem{koul02}
Koul, H. L. (2002).
\emph{Weighted Empirical Processes in Dynamic Nonlinear Models,
2nd ed.}
Lecture Notes in Statistics 166, Springer, New York, 2002.

\bibitem{kl89}
Koul, H. L. and Levental, S. (1989).
Weak convergence of the residual empirical process in explosive
autoregression.
\emph{Ann.\ Statist.}\ \tf{17}, 1784--1794.

\bibitem{levit75}
Levit, B. Y. (1975).
Conditional estimation of linear functionals.
\emph{Problems Inform.\ Transmission}\ \tf{11}, 39--54.

\bibitem{loyn80}
Loynes, R. M. (1980).
The empirical distribution function of residuals from generalised
regression.
\emph{Ann.\ Statist.}\ \tf{8}, 285--299.

\bibitem{mammen96}
Mammen, E. (1996).
Empirical process of residuals for high-dimensional linear models.
\emph{Ann.\ Statist.}\ \tf{24}, 307--335.

\bibitem{msw03}
M\"uller, U. U., Schick, A. and Wefelmeyer W. (2003).
Estimating the error variance in nonparametric regression
by a covariate-matched U-statistic.
\emph{Statistics}\ \tf{37}, 179-188.

\bibitem{msw4a}
M\"uller, U. U., Schick, A. and Wefelmeyer W. (2004a).
Estimating linear functionals of the error distribution
in nonparametric regression.
\emph{J. Statist. Plann. Inference}\ \tf{119}, 75--93.

\bibitem{msw04b}
M\"uller, U. U., Schick, A. and Wefelmeyer W. (2004b).
Estimating functionals of the error distribution in parametric 
and nonparametric regression.
\emph{J. Nonparametr. Statist.}\ \tf{16}, 525--548.

\bibitem{sw04c}
M\"uller, U. U., Schick, A. and Wefelmeyer W. (2004c).
Weighted residual-based density estimators for nonlinear autoregressive models.
To appear in: \emph{Statist.\ Sinica}.

\bibitem{owen88}
Owen, A. B. (1988).
Empirical likelihood ratio confidence intervals for a single functional.
\emph{Biometrika}\ \tf{75}, 237--249.

\bibitem{owen90}
Owen, A. B. (1990).
Empirical likelihood ratio confidence regions.
\emph{Ann.\ Statist.}\ \tf{18}, 90--120.

\bibitem{owen01}
Owen, A. B. (2001).
\emph{Empirical Likelihood}.
Monographs on Statistics and Applied Probability 92,
Chapman \& Hall/CRC, Boca Raton, FL.

\bibitem{sw02a}
Schick, A. and Wefelmeyer W. (2002a).
Estimating the innovation distribution in nonlinear
autoregressive models.
\emph{Ann. Inst. Statist. Math.}\ \tf{54}, 245-260.

\bibitem{sw02b}
Schick, A. and Wefelmeyer W. (2002b).
Efficient estimation in invertible linear processes.
\emph{Math. Methods Statist.}\ \tf{11}, 358--379.

\bibitem{sw04}
Schick, A. and Wefelmeyer W. (2004).
Root n consistent density estimators for invertible linear processes. \\
Available at: http://www.math.binghamton.edu/anton/preprint.html.

\bibitem{shor84}
Shorack, G. R. (1984).
Empirical and rank processes of observations and residuals.
\emph{Canad.\ J. Statist.}\ \tf{12}, 319--332.

\bibitem{sw86}
Shorack, G. R. and Wellner, J. A. (1986).
\emph{Empirical Processes with Applications to Statistics}.
Wiley Series in Probability and Mathematical Statistics, Wiley, New York.

\bibitem{silv78}
Silverman, B. W. (1978).
Weak and strong uniform consistency of the kernel estimate
of a density and its derivatives.
\emph{Ann.\ Statist.}\ \tf{6}, 177--184.


\end{thebibliography}
\end{document}